\documentclass[12pt,oneside,reqno]{amsart} 
\usepackage{etex}
\usepackage[T2A]{fontenc}
\usepackage[cp1251]{inputenc}
\usepackage{pifont}
\usepackage{soul}
\usepackage[dvips]{graphicx}
\usepackage{amsmath,latexsym,amsthm,cmap,indentfirst,setspace,ccaption,amssymb,color,amscd,mathrsfs}

\calclayout
\usepackage[normalem]{ulem}

\newcommand{\nicecolor}{Navy}
\usepackage[pdfauthor={J. Blanc \& E. Yasinsky}, backref=page, colorlinks, linktocpage, citecolor=\nicecolor, linkcolor=\nicecolor, urlcolor=\nicecolor]{hyperref}
\usepackage{array}

\usepackage{amssymb,lipsum}
\usepackage{float,xypic}
\usepackage{cmap,longtable}

\usepackage{geometry}
\usepackage{caption, subcaption}

\usepackage[shortlabels]{enumitem}
\setlist[1]{wide}
\setlist[2]{leftmargin=15mm}
\setlist[enumerate]{label=\rm{(\arabic*)}}
\setlist[enumerate,2]{label=\rm({\it\roman*}), }
\setlist[itemize]{label=\raisebox{0.25ex}{\tiny$\bullet$}}
\definecolor{grisclair}{rgb}{0.9,0.9,0.9}

\usepackage{tikz}
\usetikzlibrary{cd,arrows,positioning}
\tikzset{>=stealth}
\tikzcdset{arrow style=tikz}
\tikzset{link/.style={column sep=1.8cm,row sep=0.16cm}}
\tikzset{map/.style={row sep=0em, column sep=0em}}

\DeclareFontFamily{U}{mathb}{\hyphenchar\font45}
\DeclareFontShape{U}{mathb}{m}{n}{
	<5> <6> <7> <8> <9> <10> gen * mathb
	<10.95> mathb10 <12> <14.4> <17.28> <20.74> <24.88> mathb12
}{}
\DeclareSymbolFont{mathb}{U}{mathb}{m}{n}
\DeclareMathSymbol{\bigast}{1}{mathb}{"06}

\DeclareFontFamily{U}{mathx}{\hyphenchar\font45}
\DeclareFontShape{U}{mathx}{m}{n}{<-> mathx10}{}
\DeclareSymbolFont{mathx}{U}{mathx}{m}{n}
\DeclareMathAccent{\widebar}{0}{mathx}{"73}

\usepackage{verbatim}
\hypersetup{
    bookmarks=true,         % show bookmarks bar?
    unicode=true,          % non-Latin characters in Acrobat’s bookmarks
    pdftoolbar=true,        % show Acrobat’s toolbar?
    pdfmenubar=true,        % show Acrobat’s menu?
    pdffitwindow=false,     % window fit to page when opened
    pdfnewwindow=true,      % links in new PDF window
    colorlinks=true,       % false: boxed links; true: colored links
    linkcolor=blue,          % color of internal links (change box color with linkbordercolor)
    citecolor=blue,        % color of links to bibliography
    filecolor=magenta,      % color of file links
    urlcolor=cyan           % color of external links
}

\DeclareFontFamily{U}{mathx}{\hyphenchar\font45}
\DeclareFontShape{U}{mathx}{m}{n}{<-> mathx10}{}
\DeclareSymbolFont{mathx}{U}{mathx}{m}{n}
\DeclareMathAccent{\widebar}{0}{mathx}{"73}

\renewcommand{\to}{ \, \tikz[baseline=-.6ex] \draw[->,line width=.5] (0,0) -- +(.5,0); \, }

\newcommand{\longto}{ \, \tikz[baseline=-.6ex] \draw[->,line width=.5] (0,0) -- +(.9,0); \, }
\renewcommand{\longrightarrow}{\longto}
\newcommand{\rat}{ \, \tikz[baseline=-.6ex] \draw[->,densely dashed,line width=.5] (0,0) -- +(.5,0); \, }
\renewcommand{\dasharrow}{\rat}
\renewcommand{\dashrightarrow}{\rat}

\newcommand{\ps}{ \, \tikz[baseline=-.6ex] \draw[->,dotted,line width=.6] (0,0) -- +(.5,0); \,}

\renewcommand{\mapsto}{ \, \tikz[baseline=-.6ex] \draw[|->,line width=.5] (0,0) -- +(.5,0); \, }
\newcommand\iso{\stackrel{\sim}{\to}}
\newcommand{\tto}{ \, \tikz[baseline=-.6ex] \draw[->>,line width=.5] (0,0) -- +(.5,0); \, }
\renewcommand{\twoheadrightarrow}{\tto}

\DeclareMathOperator{\Bir}{Bir}

\DeclareMathOperator{\Aut}{Aut}

\DeclareMathOperator{\Ker}{Ker}
\DeclareMathOperator{\PGL}{PGL}

\DeclareMathOperator{\PSL}{PSL}

\DeclareMathOperator{\inv}{\mathbf g}

\DeclareMathOperator{\dP}{dP}

\DeclareMathOperator{\Exc}{Exc}

\DeclareMathOperator{\BirMori}{BirMori}

\theoremstyle{plain}
\newtheorem{thm}{Theorem}[section]
\newtheorem{lem}[thm]{Lemma} 
\newtheorem{cor}[thm]{Corollary} 
\newtheorem{prop}[thm]{Proposition}

\newtheorem{maintheorem}{Theorem}

\theoremstyle{definition}
\newtheorem{mydef}[thm]{Definition} 
\newtheorem{rem}[thm]{Remark}
\newtheorem*{question}{Question}
\newtheorem{ex}[thm]{Example}

\makeatletter
\g@addto@macro{\endabstract}{\@setabstract}
\newcommand{\authorfootnotes}{\renewcommand\thefootnote{\@fnsymbol\c@footnote}}%
\makeatother	

\title[cubic del Pezzo  fibrations]{Quotients of groups of birational transformations of cubic del Pezzo  fibrations}

\author{J\'er\'emy Blanc}
\address{Universit{\"a}t Basel,
Departement Mathematik und Informatik,
Spiegelgasse 1,
CH-4051, Basel,
Switzerland}
\email{Jeremy.Blanc@unibas.ch}

\author{Egor Yasinsky}
\address{Universit{\"a}t Basel,
Departement Mathematik und Informatik,
Spiegelgasse 1,
CH-4051, Basel,
Switzerland}
\email{yasinskyegor@gmail.com}

\thanks{Both authors acknowledge support by the Swiss National Science
Foundation Grant ``Birational transformations of threefolds'' 200020\_178807.}

\newcommand{\I}{\ensuremath{\mathrm{I}}}
\newcommand{\II}{\ensuremath{\mathrm{II}}}
\newcommand{\III}{\ensuremath{\mathrm{III}}}
\newcommand{\IV}{\ensuremath{\mathrm{IV}}}

\newcommand{\CC}{\mathbb C} 
\newcommand{\QQ}{\mathbb Q} 
 
\newcommand{\id}{\mathrm{id}}

\newcommand{\kk}{\Bbbk}
\newcommand{\RR}{\mathbb R}
\newcommand{\PP}{\mathbb P}
\newcommand{\ZZ}{\mathbb Z}
\newcommand{\pt}{\mathrm{pt}}

\newcommand{\Bertini}{\mathscr B}  

\begin{document}
	
	\maketitle

\begin{abstract}
We prove that the group of birational transformations of a del Pezzo fibration of degree 3 over a curve is not simple, by giving a surjective group homomorphism to a free product of infinitely many groups of order 2. As a consequence we also obtain that the Cremona group of rank 3 is not generated by birational maps preserving a rational fibration. Besides, the subgroup of $\Bir(\PP^3)$ generated by all connected algebraic subgroups is a proper normal subgroup.
\end{abstract}

\section{Introduction}\label{sec:intro}

\subsection{Non-simplicity of the Cremona groups} Let $X$ be an algebraic variety over a field $\kk$. We denote by $\Bir_\kk(X)$ its group of birational automorphisms. During the last few decades, there have been numerous papers concerning various properties of these groups. A case of particular interest is $X=\PP^n$, when the group $\Bir_\kk(\PP^n)$ is called the {\it Cremona group} of rank $n$. Rank one Cremona group over $\kk$ is just $\PGL_2(\kk)$. However, already for $n=2$ we get a much more complicated object, which is very far from being a (finite dimensional) linear algebraic group. These days, we know quite a lot about the group $\Bir_\kk(\PP^2)$, especially when the base field $\kk$ is algebraically closed of characteristic zero.  In particular, we know the presentation of this group in terms of generators and relations \cite{Gizatullin,IKT,UZ}, description of finite subgroups \cite{DI} or more generally finitely generated subgroups and other different group-theoretic properties: see \cite{Cantat13,Deserti19} for some overviews and references. Here we only want to emphasize that most questions about Cremona groups remain widely open in dimension greater than 2. 

The aim of this paper is to get some insights about the structure of $\Bir(X)$ for $X$ a three-dimensional complex algebraic variety, and in particular to address a question about the simplicity of such a group. For Cremona groups, this question goes back to F. Enriques, and it remained unsolved for more than 100 years. The non-simplicity of $\Bir_\CC(\PP^2)$ as an abstract group was first proven by S.~Cantat and S.~Lamy in \cite{CantatLamy} using an action of $\Bir_\CC(\PP^2)$ on an infinite-dimensional hyperbolic space. Later this result was generalized by A.~ Lonjou to an arbitrary base field \cite{Lonjou}. Another approach to non-simplicity of $\Bir_\RR(\PP^2)$, based on explicit presentation of $\Bir_\RR(\PP^2)$ by generators and relations, was discovered by S.~Zimmermann in \cite{Zimmermann}. The question of the non-simplicity of higher rank Cremona groups remained open until the recent preprint \cite{BLZ}, where the so-called {\it Sarkisov program} and recent achievements in the Borisov-Alexeev-Borisov conjecture \cite{BirkarA} allowed the authors to prove the following result: 

\begin{thm}\label{thm: BLZ 1}
	Let $B\subseteq \PP^m$ be a smooth projective complex variety, $P\to \PP^m$ a decomposable $\PP^2$-bundle $($projectivisation of a decomposable rank $3$ vector bundle$)$ and $X\subset P$ a smooth closed subvariety such that the projection to $\PP^m$ gives a conic bundle $\eta\colon X\to B$.
	Then there exists a group homomorphism
	\[\Bir(X)\twoheadrightarrow \bigoplus_\ZZ \ZZ/2\]
	which is surjective in restriction to $\Bir(X/B)=\{\phi\in \Bir(X)\mid \eta\circ\phi=\eta\}$. 
	
	Moreover, if there exists a subfield $\kk\subseteq \CC$ over which $X,B$ and $\eta$ are defined, the image of elements of $\Bir(X/B)$ defined over $\kk$ is also infinite.
\end{thm}

Theorem~\ref{thm: BLZ 1} applies to any product $X = \PP^1 \times B$, to smooth cubic hypersurfaces $X  \subseteq \PP^{n}$ and to many other non-rational varieties of dimension $n \ge 3$. Clearly, it also includes the case of $X = \PP^1 \times \PP^{n-1}$ which is birational to $\PP^n$, hence proving that Cremona groups $\Bir(\PP^n)$ are not simple. For a version of this result in dimension $2$ over a non-closed field, see \cite{Schneider}.

\subsection{Are there simple ``large'' groups of birational automorphisms?} In view of the previous discussion, it seems very natural to ask the following

\begin{question}[\bf S. Cantat]
	Is there a complex algebraic variety $X$ such that the group $\Bir(X)$ is ``infinite-dimensional'' and is a simple $($abstract$)$ group?
\end{question}

Let us now prove that a potential counterexample to this question must have dimension at least $3$. Firstly, the curve case is trivial. Let $X$ be an algebraic surface. We may assume that $X$ is smooth projective and minimal. If $X$ is of general type, then Matsumura's theorem \cite{Matsumura} says that $\Bir(X)$ is finite. If $X$ is birational to $\PP^2$, then $\Bir(X)$ is isomorphic to the Cremona group of rank 2, which is not simple. According to Enriques-Kodaira classification, in the remaining cases either $X$ is a non-rational ruled surface, or $K_X$ is nef. In the latter case $\Bir(X)=\Aut(X)$ and by Matsusaka's theorem \cite{Matsusaka} the group $\Aut(X)$ has a structure of a locally algebraic group with finite or countably many components. Finally, if $X$ is ruled over a non-rational curve $B$, then 
\[
\Bir(X)\simeq\PGL_2(\CC(B))\rtimes\Aut(B)
\]
which is of infinite dimension, but is not simple. Indeed, the subgroup $\PSL_2(\CC(B))\subsetneq \PGL_2(\CC(B))$ is a proper normal subgroup of $\Bir(X)$ (if $\Aut(B)$ is not trivial, one can of course also consider $\PGL_2(\CC(B))$). 

\subsection{Threefolds and Mori fibrations} Let $X$ be a rationally connected three-dimensional algebraic variety over an algebraically closed field of characteristic zero. Run the Minimal Model Program on $X$. Recall that during this program we stay in the category of projective normal varieties with at worst terminal $\QQ$-factorial singularities. Since $X$ is rationally connected, on the final step we get a {\it Fano-Mori fibration} $f: \widetilde{X}\to Z$, which means that $\dim Z<\dim X$, $Z$ is normal, $f$ has connected fibers, the anticanonical Weil divisor $-K_{\widetilde{X}}$ is ample over $Z$, and $\rho(\widetilde{X}/Z)=1$. Then we have one of the following three possibilities:

\begin{enumerate}
	\item $Z$ is a rational surface and a general fiber of $f$ is a conic;
	\item $Z\simeq\PP^1$ and a general fiber of $f$ is a smooth del Pezzo surface;
	\item $Z$ is a point and $\widetilde{X}$ is a $\QQ$-Fano threefold of Picard rank $1$.
\end{enumerate}

We see that Theorem \ref{thm: BLZ 1} covers the first case of this trichotomy (or at least most of them), so it is natural to proceed with the cases of del Pezzo fibrations and Fano threefolds. Moreover, every del Pezzo fibration of degree $\geqslant 4$ is birational to a conic bundle (see e.g. \cite[Lemma 2.5]{BCDP}), so it is natural to study the next cases, which correspond to del Pezzo fibrations of degrees $3$, $2$ or $1$. 

Our goal is to study the first case which is not covered by Theorem \ref{thm: BLZ 1} --- the case of del Pezzo fibrations of degree $3$. Our main result is the following statement:

\begin{maintheorem}\label{thm:grouphom}For each complex del Pezzo fibration $\pi\colon X\to B$ of degree~$3$ over a curve $B$, there exists a  group homomorphism 
	\[
	\Bir(X)\twoheadrightarrow\bigast\limits_{\mathbb{N}}\ZZ/2
	\]
	whose restriction to $\Bir(X/B)=\{\varphi\in \Bir(X)\mid \pi\varphi=\pi\}$ is surjective.
\end{maintheorem}

As the following example shows, there are cubic del Pezzo fibrations which are not birational to conic bundles and such that not all birational transformations preserve the cubic del Pezzo fibration. Theorem~\ref{thm:grouphom} shows then that $\Bir(V_2)$ is not simple, a fact that does not follow from \cite{BLZ}.

\begin{ex}[{\cite{Sobolev}}]\label{ex: Sobolev}
	Let $V_m\subset\PP^1\times\PP^3$ be a general divisor of bidegree $(m,3)$. Note that the variety $V_1$ is rational: the restriction of the projection $\PP^1\times\PP^3\to\PP^3$ on $V_1$ gives a birational morphism $V_1\to\PP^3$. 
	
	Consider the variety $V_2$. Let $[u_0:u_1]$ and $[x_0:x_1:x_2:x_3]$ be the homogeneous coordinates on the factors. Then $V$ is given by the equation
	\[
	A(x_0,x_1,x_2,x_3)u_0^2+B(x_0,x_1,x_2,x_3)u_0u_1+C(x_0,x_1,x_2,x_3)u_1^2=0,
	\]
	where $A,\ B$ and $C$ are polynomials of degree $3$. The system $A=B=C=0$ has 27 distinct solutions $q_i\in\PP^3$, $i=1,\ldots, 27$, which give $27$ lines $L_i=\PP^1\times q_i\subset V_2$. The Galois involution $\tau$ of the double cover $V_2\to\PP^3$ is well defined outside $L_i$, so $\tau\in\Bir(V_2)$. One can show that $V_2$ is not birationally rigid. The involution $\tau$ corresponds to the flop of the $27$ lines mentioned above, giving another structure of Mori fiber space on $V_2$ and $\tau$ is a Sarkisov link of type \IV (see \S\ref{Sec:SarkisovLinks}). In \cite{Sobolev} it is shown that $V_2$ is not birational to a conic bundle and that
	\[
	\Bir(V_2)=\langle\tau\rangle\ast\Bir(V_2/\PP^1).
	\]
\end{ex}

Note that Theorem~\ref{thm:grouphom} also applies to the group $\Bir_{\CC}(\PP^3)$, and gives a proof of the non-simplicity of $\Bir_{\CC}(\PP^3)$, alternative to the one of \cite{BLZ} (using in fact the four first sections of \cite{BLZ}, so essentially the half of the paper). Even if the quotients that we obtain are isomorphic to the ones of \cite{BLZ}, the normal subgroups of $\Bir_{\CC}(\PP^3)$ that we construct are quite different. For instance, every birational map of $\PP^3$ that preserves a conic fibration is in the kernel of our group homomorphisms. We then obtain the following result on generators of $\Bir_\CC(\PP^3)$:
\begin{maintheorem}\label{thm:TameproblemBirP3}There is a surjective group homomorphism 
	\[
	\rho\colon\Bir_\CC(\PP^3)\twoheadrightarrow\bigast\limits_{\mathbb{N}}\ZZ/2
	\]
	having the following property:
	
	Denoting by $G\subseteq \Bir_\CC(\PP^3)$ the group generated by all birational maps $\varphi\in \Bir_\CC(\PP^3)$ such that $\pi\varphi=\pi$ for some rational fibration $\pi\colon \PP^3\dasharrow \PP^2$ $($a rational map with general fibres rational$)$, we have $\Aut_\CC(\PP^3)\subsetneq G\subsetneq \Ker\rho$. In particular, we have a strict inclusion $G\subsetneq \Bir_\CC(\PP^3)$.
\end{maintheorem}

Theorem \ref{thm:TameproblemBirP3} generalises the three-dimensional version of \cite[Theorem C]{BLZ}, which proves that $\Bir_\CC(\PP^3)$ is not generated by birational maps preserving a linear fibration $\PP^3\dasharrow \PP^2$ (this corresponds to the Tame problem for $\Bir_\CC(\PP^3)$, see \cite[\S 1.C]{BLZ} for more details on this question).

We also obtain the existence of another natural proper normal subgroup of $\Bir_{\CC}(\PP^3)$ in Theorem~\ref{Theorem:AlgSubgroups} below. For each $n\ge 1$, the subgroup of $\Bir_{\CC}(\PP^n)$ generated by all connected algebraic subgroups of $\Bir_\CC(\PP^n)$ is a normal subgroup $H_n\subset \Bir_{\CC}(\PP^n)$, that contains $\Aut_{\CC}(\PP^n)\cong\PGL_{n+1}(\CC)$ (for a definition and characterisations of algebraic subgroups of $\Bir_{\CC}(\PP^n)$, see for instance \cite{BF13}). The fact that $H_n$ is normal follows from the fact that the conjugate of any connected algebraic subgroup is a connected algebraic subgroup. The equality $H_1=\Bir_{\CC}(\PP^1)$ then follows from $\Aut_{\CC}(\PP^1)=\Bir_{\CC}(\PP^1)$, and the equality $H_2=\Bir_{\CC}(\PP^2)$ follows from the Noether-Castelnuovo theorem (Lemma~\ref{Lem:NormalSubAutP2}). However, our results show that $H_3$ is a strict normal subgroup of $\Bir_{\CC}(\PP^3)$:

\begin{maintheorem}\label{Theorem:AlgSubgroups}
The subgroup of $\Bir_{\CC}(\PP^3)$ generated by all connected algebraic subgroups of $\Bir_\CC(\PP^3)$ is a strict normal subgroup of  $\Bir_{\CC}(\PP^3)$.
\end{maintheorem}

\begin{question}
	Does the statements of Theorems \ref{thm:TameproblemBirP3} and \ref{Theorem:AlgSubgroups} hold for $\Bir_\CC(\PP^n)$ with $n\geqslant 4$?
\end{question}

\subsection{The strategy of the proof and a more detailed result}
Let us briefly describe how one can obtain the group homomorphisms of Theorems~\ref{thm:grouphom} and \ref{thm:TameproblemBirP3}. According to the Sarkisov program, any birational map between two Mori fiber spaces is a composition of birational maps, called Sarkisov links. Moreover, the relations between Sarkisov links are generated by so-called {\it elementary relations} (\cite[Theorem 4.29]{BLZ}, that we reproduce here in Theorem~\ref{thm: sarkisov}). We follow the strategy of \cite{BLZ}, which studies for a Mori fiber space $X$  the {\it groupoid} $\BirMori(X)$ of all birational maps between Mori fiber spaces birational to $X$ and we construct as in \cite{BLZ} a groupoid morphism from $\BirMori(X)$ to a free product of $\ZZ/2\ZZ$ whose restriction to $\Bir(X)$ gives the desired group homomorphisms of Theorems~\ref{thm:grouphom} and \ref{thm:TameproblemBirP3}. 

The groupoid morphisms constructed in \cite{BLZ} send every Sarkisov link in $\BirMori(X)$ onto the identity, except some special Sarkisov links between conic bundles which are complicated enough. We follow the same approach, but focus on Sarkisov links between del Pezzo fibrations. The ones we are interested in are called {\it Bertini links} in the sequel (see Definition~\ref{def: bertini link}). We put an equivalence relation on such links (Definition~\ref{defi:BertinilinksequiDefi}) and associate to every Bertini link $\chi$ an integer number $\inv(\chi)$, the \emph{genus} of $\chi$ (see Definition~\ref{def: genus}), which measures the ``complexity'' of this link and only depends on the equivalence class. We then associate to each Mori fiber space $X$ the set $\Bertini_X$ of all equivalence classes of Bertini links between varieties birational to $X$ and denote by  $\Bertini_X^{\ge g}\subset \Bertini_X$ the subset of equivalence classes of Bertini links having genus $\ge g$.

The advantage of working only with Bertini links of ``high complexity'' is that they occur only in very simple elementary relations. Altogether, in Section \ref{sec: existence of hom} it allows us to prove the following statement:

\begin{maintheorem}\label{thm:BirMori}There is an integer $g\ge 0$ such that for each complex del Pezzo fibration $X\to B$ over a curve, there exists a group homomorphism 
	\[
	\Bir(X)\to\bigast\limits_{\Bertini_X^{\ge g}}\ZZ/2
	\]
	which is the restriction of a groupoid homomorphism $\BirMori(X)\to\bigast\limits_{\Bertini_X^{\ge g}}\ZZ/2$ that sends every Sarkisov link $\chi$ of Bertini type with $\inv(\chi)\ge g$ to the generator of $\ZZ/2$ indexed by equivalence class of $\chi$, and all other Sarkisov links and all automorphisms of Mori fiber spaces birational to $X$ onto the trivial element. 
	\end{maintheorem}

Theorem \ref{thm:grouphom} will follow from Theorem \ref{thm:BirMori}. Note that in Theorem \ref{thm:BirMori} there is no restriction on the degree of the del Pezzo fibration $X/B$, but we only prove that the image is non-trivial for del Pezzo fibrations of degree $3$. This is done in Section \ref{sec: image}, where we show that for a del Pezzo fibration $X/B$ of degree $3$, the image by the group homomorphism of Theorem~\ref{thm:BirMori} of a natural birational involution associated with a $2$-section of $X/B$ of genus $g'\ge g$ (where $g$ is as in Theorem \ref{thm:BirMori}) is the generator of one of the $\ZZ/2$, indexed by a Sarkisov link of genus $g'$. It then suffices to show that the genus of $2$-sections in $X/B$ is not bounded (a fact already proven in \cite{BCDP}) to show that the image of $\Bir(X)$ under the group homomorphism surjects onto a free product of $\ZZ/2$ indexed by an infinite set.

\begin{rem}
Theorem \ref{thm:BirMori}, when restricted to del Pezzo fibrations of degrees 2 and 1, gives an empty result. Indeed, there is no Bertini link on a del Pezzo fibration of degree $1$ and the genus of any Bertini link on a del Pezzo fibration of degree $2$ is equal to the genus of the base curve $B$, as it corresponds to the genus of a section blown-up. Whether $\Bir(X)$ is a simple group or not for del Pezzo fibrations of degrees 2 and 1 is an interesting question, open for the moment. The techniques developed here do not easily extend to these groups. One should consider other links and study possible relations to understand the question.
\end{rem}

\subsection{Acknowledgements} The authors would like to thank Hamid Ahmadinezhad, Ivan Cheltsov, St\'ephane Lamy and Yuri Prokhorov for interesting discussions during the preparation of this text. We would also like to thank the referees for helpful suggestions that improved the exposition of this paper.

\section{Preliminaries on the Sarkisov program}

\subsection{Notation and conventions}

Throughout the paper, we stick to the following agreements:

\begin{itemize}
	\item All the varieties and maps are defined over $\CC$, if not stated otherwise;
	\item The plain arrows ($\to$) denote morphisms; the dashed arrows ($\dashrightarrow$) denote rational maps; the dotted arrows $(\ps)$ denote pseudo-isomorphisms (isomorphisms outside codimension 2 closed subsets).
	\item We write $X\simeq Y$ if two varieties are isomorphic, and $X\approx Y$ if they are pseudo-isomorphic.
	\item A {\it birational contraction} is a birational map such that the inverse does not contract any divisor.
\end{itemize}

Below we very briefly recall some standard notions of the Minimal Model Program (MMP in the sequel) which are used in in this paper. Our main references are \cite{KM} and \cite{BCHM}; see also \cite[\S 2]{BLZ}.

\subsection{Rank r fibrations}

At the heart of this paper, as well as in \cite{BLZ}, lies the notion of {\it rank~$r$ fibration}. It puts together the notions of terminal Mori fiber space, of Sarkisov links and of elementary relations between Sarkisov links, for $r=1, 2$ and $3$ respectively. 

\begin{mydef}
	Let $r\ge 1$ be an integer.	A morphism $\eta\colon X\to B$ is a \emph{rank~$r$ fibration} if the following conditions are satisfied:
	\begin{enumerate}
		\item $X/B$ is a Mori dream space;
		\item $\dim X > \dim B \ge 0$ and $\rho(X/B) = r$;
		\item $X$ is $\QQ$-factorial and terminal, and for any divisor $D$ on $X$, the output of any $D$-MMP from $X$ over $B$ is  still $\QQ$-factorial and terminal;
		\item $B$ has klt singularities.
		\item $-K_X$ is $\eta$-big.
	\end{enumerate}
\end{mydef}

\begin{rem}\label{rem: Mori dream}
	For the definition of a Mori dream space we refer to \cite[Definition 2.3]{BLZ}, which is a relative version of \cite[Definition 1.10]{HuKeel}. Standard examples include toric varieties and Fano varieties (with $B=\pt$). For us, the reason to work with Mori dream spaces is that they behave well with respect to the MMP. Namely, if $X/B$ is a Mori dream space, then for any divisor $D$ one can run a $D$-MMP on $X$ over $B$, and there exists only finitely many possible outputs of such a process (see \cite[Proposition 1.11]{HuKeel} or \cite[Theorem 5.4]{KKL}). 
\end{rem}

Further, we say that a rank~$r$ fibration $X/B$ \emph{factorises through} a rank~$r'$ fibration $X'/B'$, or that \emph{$X'/B'$ is dominated by $X/B$}, if the fibrations $X/B$ and $X'/B'$ fit in a commutative diagram 
\[
\begin{tikzcd}[link]
X \ar[rrr] \ar[dr,dashed] &&& B \\
& X' \ar[r] & B' \ar[ur]
\end{tikzcd}
\]
where $X \rat X'$ is a birational contraction, and $B' \to B$ is a morphism with connected fibers. Note that $r \ge r'$.

\begin{prop}[{\cite[Lemma 3.3]{BLZ}}]
	Let $\eta: X\to B$ be a surjective morphism between normal varieties. Then $X/B$ is a rank $1$ fibration if and only if $X/B$ is a terminal Mori fiber space.
\end{prop}
\begin{rem}
	In threefold case, there are three possible cases for a Mori fiber space $X/B$:
	\begin{enumerate}
		\item If $\dim(B)=0$, then $X$ is a Fano variety with $\rho(X)=1$;
		\item If $\dim(B)=1$, then $X$ is a del Pezzo fibration over a curve $B$;
		\item If $\dim(B)=2$, then $X$ is a conic bundle over the surface $B$.
	\end{enumerate}
\end{rem}

In the following lemma we gather some properties of rank $r$ fibrations which we will need later.

\begin{lem}\label{lem: prop of fibrations}
	Let $r\ge 1$ and let $\eta: X\to B$ be a rank $r$ fibration. Then, the following hold:
	\begin{enumerate}
		\item\label{propfib1} If $X'$ is obtained from $X$ by making a log-flip $($resp. a divisorial contraction$)$ over $B$, then $X'/B$ is a rank $r$ fibration $($resp. a rank $(r-1)$-fibration$)$.
		\item\label{propfib2} For a general point $b\in B$ the fiber $X_b=\eta^{-1}(b)$ is pseudo-isomorphic to a weak Fano terminal variety, and the curves in $X_b$ which have a non-positive intersection with $K_{X_b}$ cover a subset of codimension at least $2$ in $X_b$.
	\end{enumerate}
\end{lem}
\begin{proof}
	\ref{propfib1} is \cite[Lemma 3.4(1)]{BLZ} and \ref{propfib2} is \cite[Corollary 3.6]{BLZ}.
\end{proof}

\subsection{Sarkisov links}\label{Sec:SarkisovLinks}

The notion of a rank $2$ fibration corresponds to the notion of Sarkisov link. Namely, a 2-ray game allows to show \cite[Lemma 3.7]{BLZ} that a rank~$2$ fibration $Y/B$ factorises through exactly two rank~$1$ fibrations $X_1/B_1, X_2/B_2$ (up to isomorphisms), which both fit into a diagram
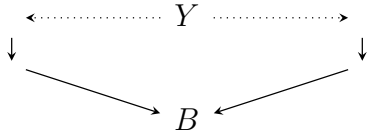
\[
\begin{tikzcd}[link]
\ar[dd]  & Y \ar[l,dotted] \ar[r,dotted] & \ar[dd] \\ \\
\ar[dr] && \ar[dl] \\
& B &
\end{tikzcd}
\]
where the top dotted arrows are sequences of log-flips, and the other four arrows are morphisms of relative Picard rank~$1$. In this situation we say that the birational map $\chi: X_1\rat X_2$ is a {\it Sarkisov link}. In the diagram above, there are two possibilities for the sequence of two morphisms on each side of the diagram: either the first arrow is a Mori fiber space, or it is a divisorial contraction and then the second arrow is a Mori fiber space. Thus, we have $4$ possibilities, which correspond to the classical definition of \emph{Sarkisov links of type \I, \II, \III{ and }\IV}, as illustrated on Figure~\ref{fig:SarkisovTypes}. Note that the varieties $Y_1,Y_2$ on the figure are $\mathbb{Q}$-factorial terminal varieties pseudo-isomorphic to $Y$, obtained via sequences of anti-flips/flops/flips, steps of the MMP.

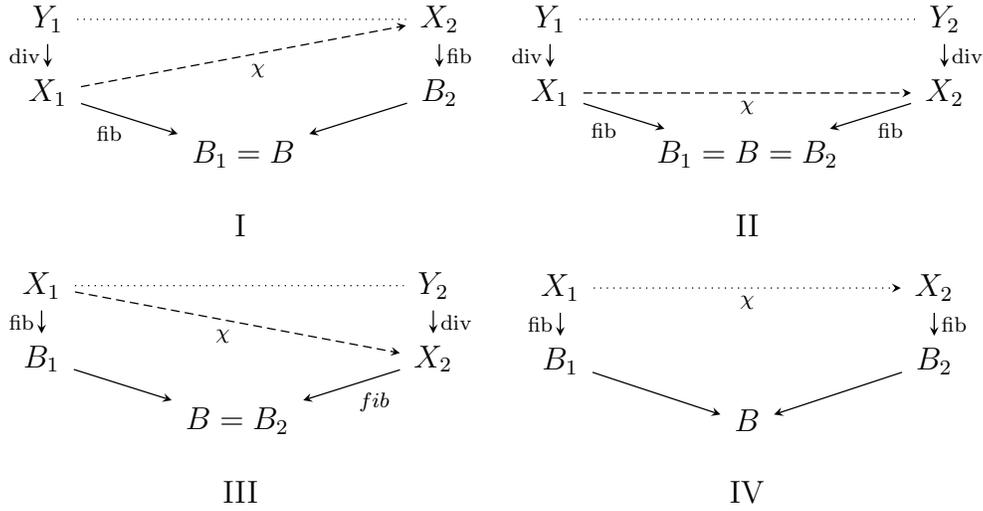
\begin{figure}[ht]
\[
{
	\def\arraystretch{2.2}
	\begin{array}{cc}
	\begin{tikzcd}[ampersand replacement=\&,column sep=1.3cm,row sep=0.16cm]
	Y_1\ar[dd,"{\rm div}",swap]  \ar[rr,dotted,-] \&\& X_2 \ar[dd,"{\rm fib}"] \\ \\
	X_1 \ar[uurr,"\chi",dashed,swap] \ar[dr,"{\rm fib}",swap] \&  \& B_2 \ar[dl] \\
	\& B_1 = B \&
	\end{tikzcd}
	&
	\begin{tikzcd}[ampersand replacement=\&,column sep=.8cm,row sep=0.16cm]
	Y_1\ar[dd,"{\rm div}",swap]  \ar[rr,dotted,-] \&\& Y_2\ar[dd,"{\rm div}"] \\ \\
	X_1 \ar[rr,"\chi",dashed,swap] \ar[dr,"{\rm fib}",swap] \&  \& X_2 \ar[dl,"{\rm fib}"] \\
	\& B_1 = B = B_2 \&
	\end{tikzcd}
	\\
	\mathrm{I} & \II
	\\
	\begin{tikzcd}[ampersand replacement=\&,column sep=1.3cm,row sep=0.16cm]
	X_1 \ar[ddrr,"\chi",dashed,swap] \ar[dd,"{\rm fib}",swap]  \ar[rr,dotted,-] \&\& Y_2\ar[dd,"{\rm div}"] \\ \\
	B_1 \ar[dr] \& \& X_2 \ar[dl,"fib"] \\
	\& B = B_2 \&
	\end{tikzcd}
	&
	\begin{tikzcd}[ampersand replacement=\&,column sep=1.7cm,row sep=0.16cm]
	X_1 \ar[rr,"\chi",dotted,swap] \ar[dd,"{\rm fib}",swap]  \&\& X_2 \ar[dd,"{\rm fib}"] \\ \\
	B_1 \ar[dr] \& \& B_2 \ar[dl] \\
	\& B \&
	\end{tikzcd}
	\\
	\III & \IV 
	\end{array}
}
\]
\caption{The four types of Sarkisov links.}
\label{fig:SarkisovTypes}
\end{figure}

The dashed arrow are pseudo-isomorphisms, the plain arrows are morphisms of relative Picard rank 1, and the arrows labeled by ``div'' are divisorial contractions. 

\subsection{Crucial example: Bertini links}

In this paper, one particular family of Sarkisov links will be especially important for us. Those are the only links which will be sent to non-trivial elements by the group homomorphism of Theorem \ref{thm:BirMori}. Let us now define this kind of links.

\begin{mydef}\label{def: bertini link}
	We say that a Sarkisov link $\chi: X_1\dashrightarrow X_2$ is of {\it Bertini type} (or $\chi$ is just a {\it Bertini link}) if 
	\begin{enumerate}
		\item $\chi$ has type $\II$ with $B$ being a curve and $X_1,X_2$ being of dimension $3$;
		\item $\eta_1$ is a blow-up of a curve $\Gamma_1\subset X_1$ such that $\Gamma_1\to B$ is surjective and the generic fiber of $Y_1/B$ is a del Pezzo surface of degree 1 (see Figure \ref{fig: SarkisovDP}).
	\end{enumerate}
\end{mydef}
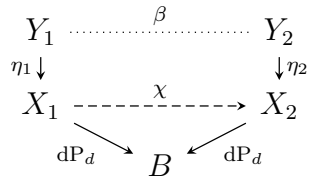
\begin{figure}[h!]
	\[
	{
		\begin{tikzcd}[ampersand replacement=\&,column sep=.8cm,row sep=0.16cm]
		Y_1\ar[dd,"\eta_1",swap]  \ar[rr,"\beta",dotted,-] \&\& Y_2\ar[dd,"\eta_2"] \\ \\
		X_1 \ar[rr,"\chi",dashed] \ar[dr,"\dP_d",swap] \&  \& X_2 \ar[dl,"\dP_d"] \\
		\& B\&
		\end{tikzcd}
	}
	\]
	\caption{A Bertini link}
	\label{fig: SarkisovDP}
\end{figure}
\begin{rem}\label{rem:BertinilinkBertiniInvolution}
The definition implies that the birational map between the generic fibers of $X_1/B$ and $X_2/B$ induced by $\chi$ is equal, up to an isomorphism of the target, to a birational involution that lifts to the Bertini involution on the generic fiber of $Y_1/B$, and implies that $X_1/B$ and $X_2/B$ are del Pezzo fibrations of the same degree $d\geqslant 2$ (follows from \cite[Theorem 2.6]{Isk1996}). This motivates the terminology. Moreover, this implies that $\eta_2$ is the blow-up of a curve $\Gamma_2$ which is birational to~$\Gamma_1$.
\end{rem}

\begin{mydef}\label{defi:BertinilinksequiDefi}
We say that two Bertini links $\chi: X_1\dashrightarrow X_2$ and $\chi': X_1'\dashrightarrow X_2'$ over $B$ and $B'$ respectively are {\it equivalent via birational maps $\psi_1$ and $\psi_2$ $($or simply equivalent$)$}, if there exists a commutative diagram 
\[\xymatrix@R=1mm@C=2cm{
	X_1\ar@{-->}[rr]^{\chi}\ar@{-->}[dd]_{\psi_1}\ar[dr] && X_2\ar@{-->}[dd]^{\psi_2} \ar[dl]  \\
	& B\ar[dd]^(.3){\psi} \\
	X_1'\ar[dr]\ar@{-->}[rr]^(.35){\chi'}&& X_2'\ar[dl]\\
	& B'&&
}\]
where for each $i\in \{1,2\}$, $\psi_i$ is a birational map inducing an isomorphism between the generic fibers of the del Pezzo fibrations $X_i/B$ and $X_i'/B'$, and where $\psi\colon B\to B'$ is an isomorphism. The equivalence class of $\chi$ will be denoted $[\chi]$. 

For each variety $X$, we denote by $\Bertini_X$ the set of equivalence classes of Bertini links $\chi: X_1\dashrightarrow X_2$, where $X_1$ and $X_2$ are birational to $X$. 
\end{mydef}

\begin{mydef}\label{def: genus}
	Let $\chi$ be a Bertini link. The {\it genus} $\inv(\chi)$ of $\chi$ is the geometric genus of the curve $\Gamma_1$ blow up by $\eta_1$:
	\[
	\inv(\chi)=\inv(\Gamma_1).
	\]
\end{mydef}
\begin{rem}\label{rem: Bertini has the same genus}
	Equivalent Bertini links have the same genus. Indeed, taking the notation of Definitions~\ref{def: bertini link} and \ref{defi:BertinilinksequiDefi}, if $\chi$ and $\chi'$ are equivalent via $\psi_1$ and $\psi_2$, then the restriction of $\psi_1$ gives a birational map between the curve $\Gamma_1\subset X_1$ blown-up by $\eta_1$ and the curve $\Gamma_1'\subset X_1'$ blown-up by $\eta_1'$.  One can thus define $\Bertini_X^{\ge g}\subset \Bertini_X$ to be the subset of equivalence classes of Bertini type links $\chi$ such that $\inv(\chi)\ge g$. 
	
	Moreover, the curves $\Gamma_1$ and $\Gamma_2$ associated to a Bertini link $\chi$ are birational (see Remark~\ref{rem:BertinilinkBertiniInvolution}), so we have $\inv(\chi)=\inv(\chi^{-1})$ for each Bertini link $\chi$.\end{rem}

\subsection{Elementary relations}

Finally, the notion of rank~$3$ fibration recovers the notion of an {\it elementary relation} between Sarkisov links, introduced in \cite{LZ} and \cite{BLZ}.

\begin{mydef}\label{def:Tequivalent}
	Let $X/B$ and $X'/B'$ be two rank~$r$ fibrations, and $T \rat X$, $T \rat X'$ two birational maps from the same variety $T$.
	We say that $X/B$ and $X'/B'$ are \emph{$T$-equivalent} (the birational maps being implicit) if there exist  a pseudo-isomorphism $X \ps X'$ and an isomorphism $B \iso B'$ such that all these maps fit in a commutative diagram:
	\[
	\begin{tikzcd}[link]
	& T \ar[dl, dashed] \ar[dr, dashed] \\
	X \ar[dd] \ar[rr,dotted] &  & X' \ar[dd] \\ \\
	B  \ar[rr,"\sim"] & & B'  \\ 
	\end{tikzcd}
	\]
\end{mydef}

As every rank~$3$ fibration $T/B$ is a Mori dream space, there are only finitely many rank~1 and rank~$2$ fibrations dominated by $T/B$ (see Remark \ref{rem: Mori dream}). Moreover, if $T/B$ factorises through a rank~1 fibration $X/B'$, then up to $T$-equivalence there are exactly two rank~$2$ fibrations which factorise through $X/B'$ and dominated by $T/B$ \cite[Lemma 4.2]{BLZ}. From these two facts one can deduce the following statement:

\begin{prop}[{\cite[Proposition 4.3]{BLZ}}]\label{pro:from T3}
	Let $T/B$ be a rank~$3$ fibration.
	Then there are only finitely many Sarkisov links $\chi_i$ dominated by $T/B$, and they fit in a relation
	\[
	\chi_t \circ \dots \circ \chi_1 = \id.
	\]
\end{prop}

\begin{mydef}\label{def:elementary relation}
	In the situation of Proposition~\ref{pro:from T3}, we say that 
	\[
	\chi_t \circ \dots \circ \chi_1 = \id
	\]
	is an \emph{elementary relation} between Sarkisov links, coming from the rank $3$ fibration $T/B$.
	Observe that the elementary relation is uniquely defined by $T/B$, up to taking the inverse, cyclic permutations and insertion of isomorphisms. 
\end{mydef}

Now we are ready to state a keynote result, which allows us to study the structure of birational automorphism groups. Its first part is basically the Sarkisov program itself, following  \cite{HMcK}. The second part is inspired by \cite[Theorem 1.3]{Kaloghiros} and proven in \cite{BLZ}. 

Let $X/B$ a Mori fiber space. We denote by $\BirMori(X)$ the {\it groupoid}\footnote{For the formalism of groupoids and of generators and relations in groupoids, we refer to \cite[\S8.2 and 8.3]{Brown}.} of birational maps between Mori fiber spaces birational to $X$. Note that $\Bir(X)$ is a subgroupoid of $\BirMori(X)$. 

\begin{thm}[{\cite[Theorem 4.29]{BLZ}}]\label{thm: sarkisov} 
	Let $X/B$ be a terminal Mori fiber space. 
	\begin{enumerate}
		\item\label{sarkisov1} The groupoid $\BirMori(X)$ is generated by Sarkisov links and automorphisms.
		\item\label{sarkisov2} Any relation between Sarkisov links in $\BirMori(X)$ is generated by elementary relations.
	\end{enumerate}
\end{thm}

\section{Elementary relations involving Bertini links}

Let $\chi_t\circ\cdots\circ\chi_1=\id$ 
be an elementary relation between Sarkisov links, coming from a rank~$3$ fibration $T/B$. Assume that one of the $\chi_i$ is a Sarkisov link $\chi_i\colon X_i\dasharrow X_{i+1}$ of type $\II$ over a curve $B_i$. Depending on whether $B$ is a point or a curve, there are only two possible ways in which $T/B$ can factorize through $\chi_i$, see Figure \ref{fig: rank 3 possibilities} where the dotted arrows and vertical arrows are pseudo-isomorphisms and divisorial contractions respectively, obtained by running an MMP, so $Y_i,\ Y_{i+1},\ T_i,\ T_{i+1}$ are then $\mathbb{Q}$-factorial terminal varieties birational to $X_i$ and $X_{i+1}$. 

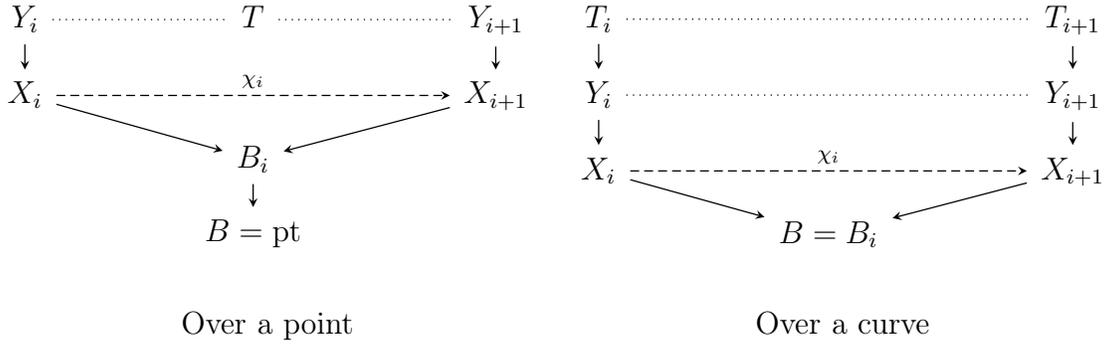
\begin{figure}[H]
	\[
	{
		\def\arraystretch{2.2}
		\begin{array}{cc}
		
		\begin{tikzcd}[link]
		Y_i\ar[dd]\ar[r,dotted,-]& T \ar[r,dotted,-] & Y_{i+1}\ar[dd]  \\ \\
		X_i \ar[rr,"\chi_i",dashed] \ar[rd]&& X_{i+1} \ar[ld]\\
		&B_i\ar[dd]\\ \\
		&B=\pt
		\end{tikzcd}
		
		&

		\begin{tikzcd}[link]
		T_i\ar[dd]\ar[rr,dotted,-] && T_{i+1}\ar[dd]  \\ \\
		Y_i\ar[dd]\ar[rr,dotted,-] && Y_{i+1}\ar[dd]  \\ \\
		X_i \ar[rr,"\chi_i",dashed] \ar[rd]&& X_{i+1} \ar[ld]\\
		&B=B_i&
		\end{tikzcd}
		\\
		\text{Over a point} & \text{Over a curve} 
		
		\end{array}
	}
	\]
	\caption{Two types of relations}
	\label{fig: rank 3 possibilities}
\end{figure} 

We call them {\it relation over a point} and {\it relation over a curve} respectively. The goal of the next two paragraphs is to show that Bertini links of large genus do not occur in non-trivial relations over a point, and occur only in very simple relations over a curve.

\subsection{Relations over a point}

In this paragraph we prove that Bertini links of large genus do not occur in non-trivial elementary relations over a point (Proposition~\ref{prop: no relations for links of type 1}).

The following statement is a consequence of the Borisov-Alexeev-Borisov conjecture, established in \cite{KMMT,Kawamata} for dimension $\leqslant 3$ and in \cite{BirkarS} for any dimension. 

\begin{prop}[{\cite[Proposition 5.1]{BLZ}}]\label{prop: BAB}
	Let $n$ be an integer, and let $\mathcal{F}_n$ be the set of weak Fano terminal varieties of dimension~$n$. There are integers $d,l,m\ge 1$, depending only on $n$, such that for each $X \in \mathcal{F}_n$ the following hold:
	\begin{enumerate}
		\item
		$h^0(-mK_X)\le l$;
		\item
		The linear system $\lvert -mK_X \rvert$ is base-point free;
		\item
		The morphism 
		\[
		\phi_{|-mK_X|}\colon X\to \PP^{h^0(-mK_X)-1}
		\] 
		is birational onto its image and contracts only curves $C\subseteq X$ with $C\cdot K_X=0$;
		\item
		$\deg \phi_{|-mK_X|}(X) \le d$.
	\end{enumerate} 
\end{prop}
\begin{cor}\label{cor: bounded genus}There is an integer $N\ge 0$ such that the following holds.
	Let $X$ be a terminal threefold, and $\eta: Y\to X$ be the blow-up of an irreducible and reduced curve $\Gamma\subset X$. Assume that both $X$ and $Y$ are pseudo-isomorphic to weak Fano terminal threefolds having a small anticanonical morphism. Then the geometric genus of $\Gamma$ satisfies $\inv(\Gamma)\le N$.
\end{cor}
\begin{proof}
	We have the following commutative diagram:
	\[\xymatrix@C+1pc{
		\widehat{X}\ar[dd]_{\phi_{|-mK_{\widehat{X}}|}} && X\ar@{.>}[ll]\ar@{.>}[ddl]^{\psi} & Y\ar[l]_{\eta}\ar@{.>}[r] & \widehat{Y}\ar[dd]^{\phi_{|-mK_{\widehat{Y}}|}}\\ 
		&& \Gamma\ar@{}[u]|{\cup}\ar@{|->}[d]\\
		\PP^a & \psi(X)\ar@{}[l]|{\supset}& \widetilde{\Gamma}\ar@{}[l]|{\supset} && \PP^b
	}\]
where $\widehat{X}$ and $\widehat{Y}$ are weak terminal Fano threefolds, and the positive integers $m$ and $a=h^0(-mK_X)-1=h^0(-mK_{\hat{X}})-1,\ b=h^0(-mK_Y)-1=h^0(mK_{\hat{Y}})-1\leqslant l$ are given by Proposition \ref{prop: BAB}. By hypothesis, the morphism $\phi_{|-mK_{\widehat{X}}|}$ is small, so the composition $\psi\colon X\ps \hat{X}\to \PP^a$ is a pseudo-isomorphism with its image $\psi(X)$, the exceptional locus $\Exc(\psi)$ is a closed subset of $X$ of dimension $\le 1$, and each irreducible curve contained in $\Exc(\psi)$ has genus $0$. As we may assume that $\inv(\Gamma)>0$, we obtain $\Gamma\nsubseteq\Exc(\psi)$. Set $\widetilde{\Gamma}=\psi(\Gamma)$. 

Since $\eta_*(mK_Y)=mK_X$, we obtain that $b\leqslant a$ and the induced rational map $\pi: \PP^a\dashrightarrow\PP^b$ is a linear projection from a linear subspace $L\subset\PP^a$. The variety $\widetilde{X}=\psi(X)\subset\PP^a$ is a threefold of degree~$\leqslant d$ by Proposition \ref{prop: BAB}, and is not contained in a hyperplane section. As in \cite[Corollary 5.2]{BLZ}, one can show that there are no irreducible surfaces $S\subseteq\widetilde{X}\cap L$. Then B\'{e}zout Theorem implies that all irreducible components of $\widetilde{X}\cap L$ of dimension $1$ have degree $\leqslant d$. Thus, $\deg(\widetilde{\Gamma})\leqslant d$.

By Castelnuovo's theorem (see e.g. \cite{Harris}), it implies that $\inv(\widetilde{\Gamma})$ is bounded by some constant $N$, depending only on $a$ and $d$, so the same holds for the geometric genus of $\widetilde{\Gamma}$, which is equal to the geometric genus of $\Gamma$.
\end{proof}
\begin{prop}\label{prop: no relations for links of type 1}
	There exists some positive integer $g$ such that no Bertini link $\chi$ with $\inv(\chi)\ge g$ occurs in a non-trivial elementary relation over a point.
\end{prop}
\begin{proof}
	From Lemma~\ref{lem: prop of fibrations} we know that $Y_1/\pt$ and $X_1/\pt$ are rank $3$ and rank $2$ fibrations respectively. Moreover, they are both pseudo-isomorphic to weak Fano terminal varieties. It remains to apply Corollary~\ref{cor: bounded genus} and take $g=N+1$.
\end{proof}

\subsection{Relations over a curve}

The main result of this paragraph is Proposition \ref{prop: relation form} which shows that if a Bertini link occurs in some elementary relation over a curve, then this relation has a very simple form. Before proving this, we need some lemmas.

\begin{lem}\label{lem: base of relation}
Let us take a Bertini link 
\[\begin{tikzcd}[ampersand replacement=\&,column sep=.8cm,row sep=0.16cm]
		Y_1\ar[dd,"\eta_1",swap]  \ar[rr,dotted,-] \&\& Y_2\ar[dd,"\eta_2"] \\ \\
		X_1 \ar[rr,"\chi",dashed] \ar[dr,"\dP_d",swap] \&  \& X_2 \ar[dl,"\dP_d"] \\
		\& B\&
		\end{tikzcd}\]
		dominated by a rank~$3$ fibration $T/B$, where $B$ is a curve, and let  $\widehat{X}/\widehat{B}$ be a rank $1$ fibration $($or equivalently a Mori fiber space$)$ such that $T/B$ factorizes through $\widehat{X}/\widehat{B}$ $($see the picture below$)$. Then the following hold:
		
		\begin{enumerate}
		\item\label{lem: base of relation1}
		The associated morphism $ \widehat{B}\to B$ is an isomorphism;
		\item\label{lem: base of relation2}The birational contraction $T\dasharrow Y_1$ contracts a divisor onto a point or a curve contained in one fiber of $Y_1/B$, and the generic fibers of $T/B$ and $Y_1/B$ are isomorphic del Pezzo surfaces of degree $1$ and Picard rank $2$;
		\item\label{lem: base of relation3}
		Exactly one of the two birational maps $X_1\dasharrow \widehat{X}$ and $X_2\dasharrow \widehat{X}$ induces an isomorphism between the generic fibers of $X_i/B_i\simeq X_i/B$ and of $\widehat{X}/B$.
		\end{enumerate}\[\begin{tikzcd}[ampersand replacement=\&,column sep=.8cm,row sep=0.16cm]
		\&  T\ar[rd,dashed]\ar[ld,dashed]\ar[drrr,dashed]\\
		Y_1\ar[dd,"\eta_1",swap]  \ar[rr,dotted,-] \&\& Y_2\ar[dd,"\eta_2"]\&\& \widehat{X}\ar[dd] \\ \\
		X_1 \ar[rr,"\chi",dashed] \ar[dr] \&  \& X_2 \ar[dl] \&\& \widehat{B}\ar[dlll] \\
		\& B\&
		\end{tikzcd}\]

\end{lem}

\begin{proof}
Let $\epsilon$ be the generic point of $B$ and let  $(X_1)_\epsilon$, $(X_2)_\epsilon$, $\widehat{X}_\epsilon$, $(Y_1)_\epsilon$, $(Y_2)_\epsilon$ and $T_\epsilon$ be the generic fibers of  $X_1/B$, $X_2/B$, $\widehat{X}/B$, $Y_1/B$, $Y_2/B$ and $T/B$ respectively. As $\chi$ is a Bertini link, the surfaces $(Y_1)_\epsilon$ and $(Y_2)_\epsilon$ are isomorphic del Pezzo surfaces of degree $1$ and Picard rank $2$, and the contractions of the two extremal rays give two morphisms $(Y_1)_\epsilon\to (X_1)_\epsilon$ and $(Y_2)_\epsilon\to (X_2)_\epsilon$ induced by $\eta_1$ and $\eta_2$.

By Lemma~\ref{lem: prop of fibrations} a general fiber of $T\to B$ is pseudo-isomorphic, and thus isomorphic, to a weak del Pezzo surface. As $(Y_1)_\epsilon$ is a del Pezzo surface of degree $1$,  the birational contraction $T\dasharrow Y_1$ contracts a divisor onto a point or a curve contained in one fiber of $Y_1/B$. In particular,  $T_\epsilon$ is isomorphic to $(Y_1)_\epsilon$. This proves \ref{lem: base of relation2}. As $T/B$ factorises through $\widehat{X}/\widehat{B}$, we have a commutative diagram
	\[
	\begin{tikzcd}[link]
	T \ar[rrr,"\pi"] \ar[dr,"\varphi",dashed,swap] &&& B \\
	& \widehat{X} \ar[r, "\widehat{\pi}",swap] & \widehat{B} \ar["\sigma",ur,swap]
	\end{tikzcd}
	\]
	where $\varphi$ is a birational contraction, and $\sigma$ is a morphism with connected fibers. As $\dim \widehat{B}\le 2$, the generic fiber $\widehat{B}_\epsilon$ of $\hat{B}/B$ is either a curve or a point, and we get morphisms of generic fibers
	\[
	T_\epsilon\overset{f}{\to} \widehat{X}_\epsilon\overset{g}{\to}\widehat{B}_\epsilon.
	\]
	The contractions of the two extremal rays of $T_\epsilon$ give $(X_1)_\epsilon$ and $(X_2)_\epsilon$, and both are del Pezzo surfaces of Picard rank $1$, so there is no morphism to a curve, and $\widehat{B}_\epsilon$ is a point, which means that $\widehat{B}\to B$ is an isomorphism (as $B$ and $\widehat{B}$ are normal), proving~\ref{lem: base of relation1}. As $\widehat{X}/\widehat{B}$ is a rank $1$ fibration, the Picard rank of $\widehat{X}_\epsilon$ is $1$, so $f$ corresponds to one of the contractions of the two extremal rays of $T_\epsilon$. This gives~\ref{lem: base of relation3}.
\end{proof}

We also need the following easy lemma.

\begin{lem}[{\cite[Lemma 2.22]{BLZ}}]\label{lem: 2rays2Divisors}
	Let $T \to Y$ and $Y\to X$ be two divisorial contractions between $\QQ$-factorial varieties, with respective exceptional divisors $E$ and $F$.
	Assume that there exists a morphism $X \to B$ such that $T/B$ is a Mori dream space.
	Then there exist two others $\QQ$-factorial varieties $T'$ and $Y'$, with a pseudo-isomorphism $T \ps T'$ and birational contractions $T' \to Y' \to X$, with respective exceptional divisors the strict transforms of $F$ and $E$, such that the following diagram commutes:
	\[
	\begin{tikzcd}[link]
	T \ar[rr, dotted, -] \ar[dd,"E",swap] && T' \ar[dd,"F"] \\ \\
	Y \ar[dr,"F",swap] && Y' \ar[dl,"E"] \\
	& X &
	\end{tikzcd}
	\]
\end{lem}

We can now prove the main statement of this paragraph. 

\begin{prop}\label{prop: relation form}
	Let $\chi_1: X_1/B\dashrightarrow X_2/B$ be a Bertini link appearing in some non-trivial elementary relation over a curve. Then this relation has the form
	\[
	\chi_4\circ\chi_3\circ\chi_2\circ \chi_1=\id,
	\]
	where the Sarkisov links $\chi_2\colon X_2\dasharrow X_3$, $\chi_3\colon X_3\dasharrow X_4$, $\chi_4\colon X_4\dasharrow X_1$ are between Mori fibre spaces $X_i/B$, $i=1,2,3,4$, such that $\chi_3$ is a Bertini link and $\chi_2$, $\chi_4$ induce isomorphisms between the generic fibres of $X_i/B$. In particular, the Bertini links $\chi_1$ and ${\chi_3}^{-1}$ are equivalent via ${\chi_4}^{-1}$ and $\chi_2$.
\end{prop}
\begin{proof}
	Let $T/B$ be a rank~$3$ fibration corresponding to an elementary relation $\chi_n\circ\cdots\circ\chi_1=\id$. Denote by $\{X_i/B_i\}_{i=1}^n$ the finite collection of rank~1 fibrations dominated by $T/B$ and corresponding to Sarkisov liks $\chi_i: X_i\dashrightarrow X_{i+1}$. We are going to show that $n=4$.
	
	First, Lemma~\ref{lem: base of relation}\ref{lem: base of relation1} shows that the morphism $B_i\to B$ is an isomorphism for each $i$. In particular, all $\chi_i$ are Sarkisov links of type $\II$ over $B$. We will show that the links are given as in the diagram below, where the varieties associated with $\chi_i$ are arranged in circles according to their Picard number and where the morphisms are labelled by divisors which those morphisms contract.
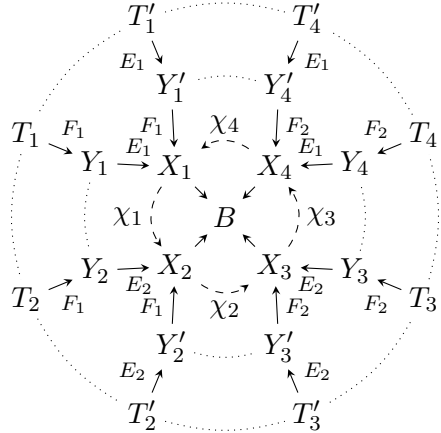
\[
\begin{tikzpicture}[scale=1.1,font=\small,outer sep=-1pt]
\node (Y8) at (112.5:1.7cm) {$Y_1'$};
\node (T8) at (112.5:2.6cm) {$T_1'$};
\node (Y1) at (157.5:1.7cm) {$Y_1$};
\node (T1) at (157.5:2.6cm) {$T_1$};
\node (Y2) at (202.5:1.7cm) {$Y_2$};
\node (T2) at (202.5:2.6cm) {$T_2$};
\node (Y3) at (247.5:1.7cm) {$Y_2'$};
\node (T3) at (247.5:2.6cm) {$T_2'$};
\node (Y4) at (292.5:1.7cm) {$Y_3'$};
\node (T4) at (292.5:2.6cm) {$T_3'$};
\node (Y5) at (337.5:1.7cm) {$Y_3$};
\node (T5) at (337.5:2.6cm) {$T_3$};
\node (Y6) at (382.5:1.7cm) {$Y_4$};
\node (T6) at (382.5:2.6cm) {$T_4$};
\node (Y7) at (427.5:1.7cm) {$Y_4'$};
\node (T7) at (427.5:2.6cm) {$T_4'$};
\node (B) at (0:0cm) {$B$};
\foreach \x in {1,...,4}{
\node (X\x) at (\x*90+45:0.85cm) {$X_{\x}$};
\draw[->] (X\x) to (B);
}
\draw[dotted,-] (T1) to [bend right=15] (T2)
		(T2) to [bend right=15] (T3)
		(T3) to [bend right=15] (T4)
		(T4) to [bend right=15] (T5)
		(T5) to [bend right=15] (T6)
		(T6) to [bend right=15] (T7)
		(T7) to [bend right=15] (T8)
		(T8) to [bend right=15] (T1)
		(Y1) to [bend right=15] (Y2)
		(Y3) to [bend right=15] (Y4)
		(Y5) to [bend right=15] (Y6)
		(Y7) to [bend right=15] (Y8);
\draw[dashed,->] (X1) to [bend right=35,swap,"$\chi_1$"] (X2);
\draw[dashed,->] (X2) to [bend right=35,swap,"$\chi_2$"] (X3);
\draw[dashed,->] (X3) to [bend right=35,swap,"$\chi_3$"] (X4);
\draw[dashed,->] (X4) to [bend right=35,swap,"$\chi_4$"] (X1);
\draw[->] (Y1) to [pos=0,"\tiny $E_1$"] (X1);
\draw[->] (Y2) to [swap,pos=0,"\tiny $E_2$"] (X2);
\draw[->] (Y3) to [pos=0,"\tiny $F_1$"] (X2);
\draw[->] (Y4) to [swap,pos=0,"\tiny $F_2$"] (X3);
\draw[->] (Y5) to [pos=0,"\tiny $E_2$"] (X3);
\draw[->] (Y6) to [swap,pos=0,"\tiny $E_1$"] (X4);
\draw[->] (Y7) to [pos=0,"\tiny $F_2$"] (X4);
\draw[->] (Y8) to [swap,pos=0,"\tiny $F_1$"] (X1);
\draw[->] (T1) to [pos=0.2,"\tiny $F_1$"] (Y1);
\draw[->] (T2) to [swap,pos=0.2,"\tiny $F_1$"] (Y2);
\draw[->] (T3) to [pos=0.2,"\tiny $E_2$"] (Y3);
\draw[->] (T4) to [swap,pos=0.2,"\tiny $E_2$"] (Y4);
\draw[->] (T5) to [pos=0.2,"\tiny $F_2$"] (Y5);
\draw[->] (T6) to [swap,pos=0.2,"\tiny $F_2$"] (Y6);
\draw[->] (T7) to [pos=0.2,"\tiny $E_1$"] (Y7);
\draw[->] (T8) to [swap,pos=0.2,"\tiny $E_1$"] (Y8);
\end{tikzpicture}
\]
	 We start with the link $\chi_1$ and try to recover the whole relation. Let $E_1\subset Y_1$ and $E_2\subset Y_2$ be exceptional divisors of the divisorial contractions $Y_1\to X_1$ and $Y_2\to X_2$ respectively. As $\chi_1$ is a Bertini link, both are contracted onto curves of $X_1$ and $X_2$ respectively, which are not contained in one fibre. We again denote by $E_1,\ E_2\subset T$ the strict transforms of these divisors on $T$. We denote by $F_1$ the divisor contracted by the divisorial contraction $T_1\to Y_1$. By Lemma~\ref{lem: base of relation}\ref{lem: base of relation2}, the generic fibers of $T_1/B$ and $Y_1/B$ are isomorphic, so the image of $F_1$ is a point or a curve contained in a fiber of $Y_1/B$. The same then holds for $T_2\to Y_2$, and we obtain the following part of the above diagram.
	 
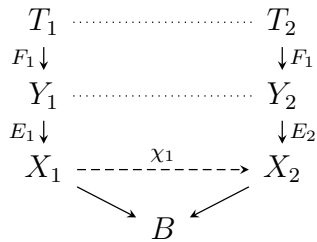
\[\begin{tikzcd}[ampersand replacement=\&,column sep=.8cm,row sep=0.16cm]
		T_1\ar[dd,"F_1",swap]  \ar[rr,dotted,-] \&\& T_2\ar[dd,"F_1"] \\ \\
		Y_1\ar[dd,"E_1",swap]  \ar[rr,dotted,-] \&\& Y_2\ar[dd,"E_2"] \\ \\
		X_1 \ar[rr,"\chi_1",dashed] \ar[dr,swap] \&  \& X_2 \ar[dl] \\
		\& B\&
		\end{tikzcd}\]
		Denote by $F_2\subset Y_1$ the fiber containing the image of $F_1$. As before, we again write $F_2$ for the strict transforms of $F_2$ in the threefolds $X_i,Y_i,T_i$. 
		
		For each integer $i$, the rank $3$ fibration $T_i/B$ factorises through $X_i/B_i=X_i/B$.  We obtain the following key facts:
	\begin{enumerate}
		\item The morphism $T_i\to X_i$ contracts {\it exactly one} divisor among $F_1$ and $F_2$. Indeed, each fiber of our del Pezzo fibration $X_i\to B$ is an {\it irreducible} surface.
		\item The morphism $T_i\to X_i$ contracts {\it exactly one} divisor among $E_1$ and $E_2$. This is because  exactly one of the two birational maps $X_1\dasharrow X_i$ and $X_2\dasharrow X_i$ induces an isomorphism between the generic fibers of $X_1/B$ or $X_2/B$ with $X_i/B$ (Lemma~\ref{lem: base of relation}\ref{lem: base of relation3}).
	\end{enumerate}
	We use these to finish the proof. Let us consider then next link $\chi_2$ that starts from $X_2$. Its resolution is given by 
	 \[\begin{tikzcd}[ampersand replacement=\&,column sep=.8cm,row sep=0.16cm]
		T_2'\ar[dd,swap]  \ar[rr,dotted,-] \&\& T_3'\ar[dd] \\ \\
		Y_2'\ar[dd,swap]  \ar[rr,dotted,-] \&\& Y_3'\ar[dd] \\ \\
		X_2 \ar[rr,"\chi_2",dashed] \ar[dr,swap] \&  \& X_3 \ar[dl] \\
		\& B\&
		\end{tikzcd}\]
	
	As $T_2\approx T_2'$ over $X_2$ and the birational map $\chi_2\circ \chi_1$ should not be an isomorphism, the morphism $Y_2'\to X_2$ must contract $F_1$, and $T_2'\to Y_2'$ must contract $E_2$ (this exists by Lemma~\ref{lem: 2rays2Divisors}). As $T_2\approx T_3'$ and $Y_2'\approx Y_3'$, the morphism $T_3'\to Y_3'$ contracts $E_2$. Then (2) implies that $Y_3'\to X_3$ contracts $F_2$.
	
	Proceeding this game further, we recover all labels and prove $n=4$. The next step consists of studying $\chi_3$. We obtain a commutative diagram

	 \[\begin{tikzcd}[ampersand replacement=\&,column sep=.8cm,row sep=0.16cm]
		T_2'\ar[dd,"E_2",swap]  \ar[rr,dotted,-] \&\& T_3'\ar[dd,"E_2"] \ar[rr,dotted,-] \&\& T_3\ar[dd] \ar[rr,dotted,-] \&\& T_4\ar[dd]\\ \\
		Y_2'\ar[dd,"F_1",swap]  \ar[rr,dotted,-] \&\& Y_3'\ar[ddr,"F_2"]\&\& Y_3 \ar[ddl]\ar[rr,dotted,-] \&\&Y_4\ar[dd]\\\ \\
		X_{2} \ar[rrr,"\chi_{2}",dashed] \ar[drrr,swap] \&\&\& X_3\ar[rrr,"\chi_{3}",dashed]\ar[d]   \&\&\& X_{4}\ar[dlll] \\
		\&\&\&B\&
		\end{tikzcd}\]
	As before, the fact that $T_3\approx T_3'$ implies that $T_3\to X_3$ contracts $E_2$ and $F_2$. This has to be in a different order than $T_3'\to X_3$ (since otherwise $\chi_3\circ \chi_2$ would be the identity), so $T_3\to Y_3$ contracts $F_2$ and $Y_3\to X_3$ contracts $E_2$; these contractions exist again by Lemma~\ref{lem: 2rays2Divisors}. As above, the fact that $T_3\approx T_4$ and $Y_3\approx Y_4$ implies that $T_4\to Y_4$ contracts $F_2$, and then $Y_4\to X_4$ contracts $E_1$ by $(2)$. The same argument applied to the next link $\chi_4$ gives a commutative diagram
	 \[\begin{tikzcd}[ampersand replacement=\&,column sep=.8cm,row sep=0.16cm]
		T_3\ar[dd,"F_2",swap]  \ar[rr,dotted,-] \&\& T_4\ar[dd,"F_2"] \ar[rr,dotted,-] \&\& T_4'\ar[dd,"E_1"] \ar[rr,dotted,-] \&\& T_5'\ar[dd,"E_1"]\\ \\
		Y_3\ar[dd,"E_2",swap]  \ar[rr,dotted,-] \&\& Y_4\ar[ddr,"E_1"]\&\& Y_4' \ar[ddl,"F_2",swap]\ar[rr,dotted,-] \&\&Y_5'\ar[dd,"F_1"]\\\ \\
		X_{3} \ar[rrr,"\chi_{3}",dashed] \ar[drrr,swap] \&\&\& X_4\ar[rrr,"\chi_{4}",dashed]\ar[d]   \&\&\& X_{5}\ar[dlll] \\
		\&\&\&B\&
		\end{tikzcd}\]
		As $T_5'\approx T_1$ and both $T_1\to X_1$ and $T_5'\to X_5$ contract the two divisors $E_1$ and $F_1$, there is an isomorphism $X_5\iso X_1$ compatible with the morphisms $T_1\to X_1$ and $T_5'\to X_5$ (see also Lemma~\ref{lem: 2rays2Divisors}). In particular, we may replace $X_5$ with $X_1$ in the above diagram and then simply write $T_5'=T_1'$ and $Y_5'=Y_1'$. The rational map $T_1\dasharrow X_1$ obtained here is the same as the initial one, so we have covered all possibilities, and obtain $n=4$, together with the desired diagram.

	As $Y_3\to X_3$ contract $E_2$, the link $\chi_3$ is a Bertini link (see Definition~\ref{def: bertini link}).
	As only the divisor $E_2$ (and not $E_1$) is contracted by $T_2'\to X_2$ and $T_3'\to X_3$, the link $\chi_2$ induces an isomorphism between the generic fibers of $X_2\to B$ and $X_3\to B$. Similarly, only $E_1$ (and not $E_2$) is contracted by $T_4'\to X_4$ and $T_1'\to X_1$, so $\chi_4$  induces an isomorphism between the generic fibers of $X_4\to B$ and $X_1\to B$. This means that $\chi_1$ is equivalent to ${\chi_3}^{-1}$ via  $\chi_4^{-1}$ and $\chi_2$ (Definition~\ref{defi:BertinilinksequiDefi}).

\end{proof}

\subsection{Proof of Theorem \ref{thm:BirMori}}\label{sec: existence of hom}

We are now ready to prove Theorem \ref{thm:BirMori}. We use the presentation of the groupoid $\BirMori$ given in Theorem \ref{thm: sarkisov} and the description of elementary relations involving Bertini links, obtained in Propositions \ref{prop: no relations for links of type 1} and \ref{prop: relation form}.

\begin{proof}[Proof of Theorem \ref{thm:BirMori}]
	Let $X$ be a del Pezzo fibration of degree $3$. By Theorem~\ref{thm: sarkisov}\ref{sarkisov1}, the groupoid $\BirMori(X)$ is generated by Sarkisov links and automorphisms of Mori fiber spaces. By Theorem~\ref{thm: sarkisov}\ref{sarkisov2}, the relations are generated by elementary relations. Fix a positive integer $g\in\ZZ_{>0}$ given by Proposition \ref{prop: no relations for links of type 1}. Consider the map
	\[
	\Phi: \BirMori(X)\to \bigast\limits_{\Bertini_X^{\ge g}} \ZZ/2
	\] 
 	which sends each Sarkisov link $\chi$ of Bertini type with $\inv(\chi)\ge g$ to the generator of $\ZZ/2$ indexed by $[\chi]$, and all other Sarkisov links and automorphisms of Mori fiber spaces are sent to zero. To check that $\Phi$ is a well-defined groupoid homomorphism, we need to show that every elementary relation is sent to the neutral element.
 	
	Let $\chi_n\circ\cdots\circ\chi_1=\id$ 
	be a non-trivial elementary relation between the Sarkisov links. We may assume that it involves a Bertini link $\chi_1$ with $\inv(\chi_1)\ge g$, otherwise the relation is sent to the neutral element automatically. By Proposition~\ref{prop: no relations for links of type 1}, our relation is a relation over a curve. Then Proposition~\ref{prop: relation form} shows that $n=4$. Moreover, $\chi_1$ and $\chi_3$ are equivalent links of Bertini type, while $\chi_2$ and $\chi_4$ are not of Bertini type; this latter fact follows from the fact that $\chi_2$ and $\chi_4$ are isomorphisms between the generic fibres $X_i\to B$, by Proposition~\ref{prop: relation form}, and thus cannot be of Bertini type, by Remark~\ref{rem:BertinilinkBertiniInvolution}. Thus, our elementary relation is sent to the neutral element. This proves the existence of the groupoid homomorphism. The restriction on $\Bir(X)$ gives a group-theoretic homomorphism.
\end{proof}

\section{Image of the group homomorphism given by Theorem \ref{thm:BirMori}}\label{sec: image}

In this section, we prove that the image of the group homomorphism given by Theorem \ref{thm:BirMori} is large when $X$ is a cubic del Pezzo fibration. Recall that a Bertini involution on a cubic surface $F\subset \PP^3$ associated with two general points $p,q\in F$ is given by $\pi\iota\pi^{-1}\in \Bir(X)$, where $\pi\colon \widehat{F}\to F$ is the blow-up of the two points $p,q$ and $\iota\in \Aut(\widehat{F})$ is the Bertini involution of the del Pezzo surface $\widehat{F}$ of degree $1$, given by the double covering $\widehat{F}\stackrel{\lvert -2K_{\widehat F}\rvert}{\longrightarrow} V\subset \PP^3$ onto a quadric cone $V\subset \PP^3$. The following result does this in family on a cubic del Pezzo fibration.

\begin{prop}[{\cite[Proposition 5.3]{BCDP}}]\label{prop: genus is unbounded}
	Let $X/B$ be a del Pezzo fibration of degree $3$. For each positive integer $m>0$ there exists a smooth $2$-section $\Gamma\subset X$ with genus $\inv(\Gamma)\ge m$, and a birational involution $\iota_\Gamma$ which acts on a general fibre $F$ of $X/B$ as the Bertini involution of the cubic surface $F$, centered at the two points of $F\cap \Gamma$.
\end{prop}
\begin{proof}[Sketch of proof]
	We may assume that $X$ is brought to a standard model. Take a section $Z\subset X$, and let $\eta: Y\to X$ be its blow-up (note that $Z$ exists by \cite[IV, Theorem 6.8]{Kollar_rational}). Then $Y$ is a del Pezzo fibration of degree $2$. Denote by $F_X$ a general fiber of $X\to\PP^1$, and by $F_Y$ a general fiber of $Y\to\PP^1$. For some large integer $n$, consider the divisor
	\[
	D=-K_Y+nF_Y.
	\]
	One can show that $D$ is base-point-free, for $n$ big enough. Now take two general elements $D_1,\ D_2\in |D|$, and consider the curve $\Delta=D_1\cap D_2$. Then $Q$ is a smooth curve contained in the smooth locus of $Y$.
	
	As $D_1,\ D_2$ correspond to the strict transforms of two hyperplanes in $\PP^3$ through the point $Z\cap F_X$, we get that $F_Y\cap\Delta$ consists of 2 points, i.e.~$\Delta$ is a 2-section. Now let $\Gamma=\eta(\Delta)$, which is a $2$-section of $X/B$. Straightforward calculations show that 
	\[
	\deg(-K_\Gamma)=4n-K_Y^3,
	\]
	which proves that the genus of $\Gamma$ is bigger than $m$ if one chooses $n$ large enough. The birational involution $\iota_\Gamma\in \Bir(X)$ is then obtained by applying to a general fibre $F$ of $X/B$ the Bertini involution of the cubic surface $F$, centered at  the two points of $F\cap \Gamma$.
\end{proof}

To deduce Theorem \ref{thm:grouphom} and Theorem \ref{thm:TameproblemBirP3} from Theorem~\ref{thm:BirMori},  we will show that the involutions $\iota_\Gamma$ of Proposition~\ref{prop: genus is unbounded} are sent, via the group homomorphism of Theorem~\ref{thm:BirMori}, onto infinitely many different generators of $\bigast\limits_{\Bertini_X^{\ge g}}\ZZ/2$.

\begin{lem}\label{lem: relative Sarkisov}
	Let $X/B$ and $Y/B$ be two Mori fiber spaces over the same base, and $\varphi: X\dashrightarrow Y$ be a birational map over $B$. There exists a decomposition of $\varphi$ into a product of Sarkisov links 
	\[
	\varphi=\chi_1\circ\cdots\circ\chi_r,
	\]
	where $\chi_i\colon X_i/B_i\dashrightarrow X_{i+1}/B_{i+1}$ and the following hold:
	\begin{enumerate}
		\item\label{relative Sarkisov1} If $\dim(B)=2$, one can  assume that $\dim(B_{i})=2$ for each $i$;
		\item\label{relative Sarkisov2} If $X/B$ is a del Pezzo fibration of degree $d\leqslant 3$ over a curve $B$, then one can choose each $\chi_i$ to be a Sarkisov link of type \II, each $B_i$ equal to $B$, and all $X_i/B_i$ to be del Pezzo fibrations of the same degree $d$ $($here we set $X=X_1$, $Y=X_{r+1})$;
	\end{enumerate} 
\end{lem}
\begin{proof}
	We first note that the proof of \cite[Theorem 1.1]{HMcK} works in relative settings, i.e.~the map $\varphi$ can be decomposed into a product of Sarkisov links $\chi_i$ {\it over the fixed base $B$}. This means that all intermediate varieties admit morphisms to that $B$, as in the figure below.

	
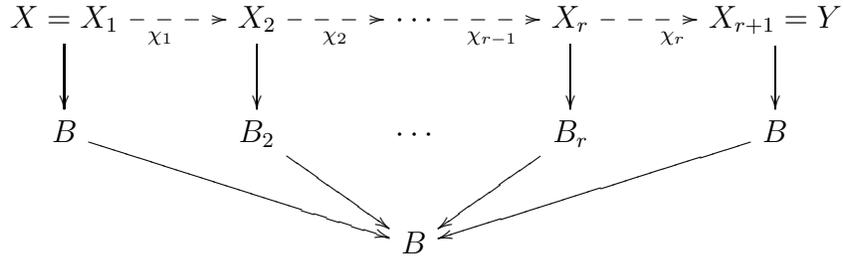
\[\xymatrix@C+1pc{
		X=X_1\ar@{-->}[r]_{\chi_1}\ar[d] & X_2\ar@{-->}[r]_{\chi_2}\ar[d] & \cdots\ar@{-->}[r]_{\chi_{r-1}} &  X_{r}\ar@{-->}[r]_{\chi_r}\ar[d] & X_{r+1}=Y\ar[d]\\
		B\ar[drr] & B_2\ar[dr] & \cdots  &  B_{r}\ar[dl] & B\ar[dll]\\
		& & B & &
	}\]

	If $\dim B=2$, then all $B_i$ are surfaces, proving~\ref{relative Sarkisov1}. To prove \ref{relative Sarkisov2}, we suppose that $X/B$ is a del Pezzo fibration of degree $d\leqslant 3$. We denote by $(X_i)_{\epsilon}$ the generic fibre of $X_i/B$, for each $i$, and obtain a birational map $(X_1)_{\epsilon}\dasharrow (X_i)_{\epsilon}$ for each $i$. As $(X_1)_{\epsilon}$ is a del Pezzo surface of degree $d\le 3$ with Picard rank equal to $1$, the same holds for $(X_i)_{\epsilon}$ (follows from the classification of Sarkisov links between two-dimensional Mori fibrations over a perfect field given in \cite[Theorem 2.6]{Isk1996}). This implies that the generic fibre $(B_i)_\epsilon$ of $B_i\to B$ is a point and thus $B_i\to B$ is an isomorphism. The Sarkisov link $\chi_i$ is then of type \II. This achieves the proof of~\ref{relative Sarkisov2}.
\end{proof}

\begin{lem}\label{lemm: number of Bertini is odd}Let $g\ge 0$ be an integer as in Theorem~\ref{thm:BirMori}.
	 Let $X/B$ be a cubic del Pezzo fibration over a curve $B$, and $\iota_\Gamma\in\Bir(X)$ be the birational involution associated with a smooth  $2$-section $\Gamma$ of $X/B$ with $\inv(\Gamma)\ge g$, as in Proposition~$\ref{prop: genus is unbounded}$. Then, the image of $\iota_\Gamma$ under the group homomorphism
	\[
	\Bir(X)\to\bigast\limits_{\Bertini_X^{\ge g}}\ZZ/2
	\]
	of Theorem~\ref{thm:BirMori} is the generator of a group $\ZZ/2$ indexed by the equivalence class of a Bertini link of genus $\inv(\Gamma)$.
\end{lem}
\begin{proof}
By Lemma~\ref{lem: relative Sarkisov}\ref{relative Sarkisov2}, the map $\varphi$ can be decomposed into the product of Sarkisov links $\chi_1,\ldots,\chi_r$ between cubic del Pezzo fibrations over $B$. The decomposition given in the proof of \cite[Theorem 1.3]{HMcK} comes from birational contractions $\pi_i: Z\dashrightarrow X_i$, where $Z$ is the resolution of indeterminacy points of $\varphi$.
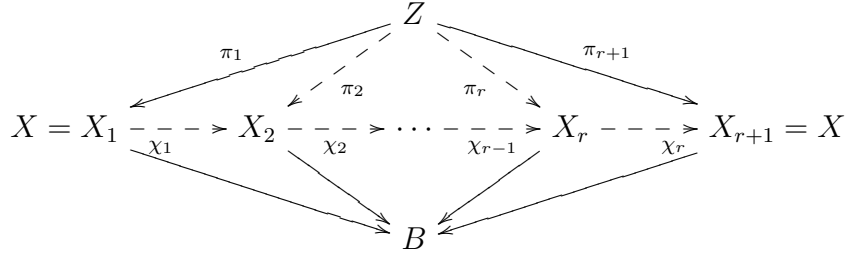
		\[\xymatrix@C+1pc{
		& & Z\ar[dll]_{\pi_1}\ar@{-->}[dl]^{\pi_2}\ar@{-->}[dr]_{\pi_r}\ar[drr]^{\pi_{r+1}} & &\\ 
		X=X_1\ar@{-->}[r]_{\chi_1}\ar[drr] & X_2\ar@{-->}[r]_{\chi_2}\ar[dr] & \cdots\ar@{-->}[r]_{\chi_{r-1}} &  X_{r}\ar@{-->}[r]_{\chi_r}\ar[dl] & X_{r+1}=X\ar[dll]\\
		& & B & &
	}\]
Denote by $S_1\subset X_1$ and $S_{r+1}\subset X_{r+1}$ the divisors contracted by $\varphi$ and $\varphi^{-1}$ respectively. Further, let $E_i\subset Z$ be $\pi_i$-exceptional divisor over $S_i$, where $i\in \{1,r+1\}$. Note that the generic fiber of $Z/B$ is a del Pezzo surface of degree 1 and Picard rank $2$ over $\CC(B)$. Thus each $\pi_i$ contracts precisely one of $E_1$ and $E_{r+1}$, and possibly some ``vertical'' divisors contained in the fibers of $Z/B$. If $\pi_i$ and $\pi_{i+1}$ contract the same divisor $E_j$, the Sarkisov link $\chi_i$ induces an isomorphism between the generic fibers of $X_i/B$ and $X_{i+1}/B$. If the contracted divisors are different, the corresponding link is a Bertini link. Since the first and the last morphisms (i.e.~$\pi_1$ and $\pi_{r+1}$) contract different divisors, the number of Bertini links must be odd. Moreover, all Bertini links have the same genus~$\inv(\Gamma)$.
\end{proof}

We now can complete the proof of Theorem \ref{thm:grouphom}.

\begin{proof}[Proof of Theorem \ref{thm:grouphom}]
	Let $X/B$ be a cubic del Pezzo fibration and let us write \[\Bir(X/B)=\{\varphi\in \Bir(X)\mid \pi\varphi=\pi\}.\] By Theorem \ref{thm:BirMori}, there exists an integer $g\geqslant0$ and a group homomorphism 
	\[
	\rho\colon \Bir(X)\to\bigast\limits_{\Bertini_X^{\ge g}}\ZZ/2
	\]
	which sends every Sarkisov link $\chi$ of Bertini type with $\inv(\chi)\ge g$ to the generator indexed by equivalence class of $\chi$, and all other Sarkisov links and all automorphisms of Mori fiber spaces birational to $X$ onto the trivial element. For each integer $g'\ge g$, Proposition~\ref{prop: genus is unbounded} gives a birational involution $\iota_\Gamma\in \Bir(X/B)$ associated to a curve of genus $\inv(\Gamma)\ge g'$, and whose image by $\rho$ is the generator of a Sarkisov link of genus $\inv(\Gamma)$ (Lemma~\ref{lemm: number of Bertini is odd}). We obtain then infinitely many distinct involutions $\iota_\Gamma\in \Bir(X)$  that are sent onto infinitely many distinct generators of $\bigast\limits_{\Bertini_X^{\ge g}}\ZZ/2$. Projecting on the corresponding factors gives a group homomorphism $\Bir(X)\twoheadrightarrow\bigast\limits_{\mathbb{N}}\ZZ/2$ whose restriction to $\Bir(X/B)$ is surjective.
\end{proof}

\subsection{Proof of Theorems \ref{thm:TameproblemBirP3} and ~\ref{Theorem:AlgSubgroups}}

We now use the group homomorphisms constructed in Theorems \ref{thm:grouphom} and \ref{thm:BirMori} and the relative Sarkisov program to prove Theorem \ref{thm:TameproblemBirP3} and~\ref{Theorem:AlgSubgroups}. We work as before over the field of complex numbers, and thus write $\Aut(\PP^n)=\Aut_{\CC}(\PP^n)$ and $\Bir(\PP^n)=\Bir_{\CC}(\PP^n)$.

\begin{proof}[Proof of Theorem \ref{thm:TameproblemBirP3}]
	
Let $X/B$ be a del Pezzo fibration of degree $3$ such that $X$ is rational (one can take e.g.~$V_1$ from Example \ref{ex: Sobolev}). Fixing a birational map $\psi: X\dashrightarrow\PP^3$ gives an isomorphism $\Bir(\PP^3)\simeq \Bir(X)$. We then obtain an surjective group homomorphism 
\[
\rho\colon \Bir(\PP^3)\simeq\Bir(X)\twoheadrightarrow\bigast\limits_{\mathbb{N}}\ZZ/2
\]
given by Theorem~\ref{thm:grouphom}. By construction, it is obtained by composing the group homomorphism $\hat\rho\colon\Bir(X)\to\bigast\limits_{\Bertini_X^{\ge g}}\ZZ/2$ of Theorem~\ref{thm:BirMori} with a projection $\bigast\limits_{\Bertini_X^{\ge g}}\ZZ/2\to \bigast\limits_{\mathbb{N}}\ZZ/2$. Let $\varphi\in\Bir(\PP^3)$ be a birational map such that $\pi\varphi=\pi$ for some rational fibration $\pi\colon \PP^3\dashrightarrow\PP^2$ (a rational map with general fibres being rational curves). We want to show that $\varphi\in\Ker\rho$, which corresponds to ask that $\varphi'=\psi^{-1} \varphi\psi$ is in the kernel of $\hat\rho$.

Wrting $\pi'=\pi\psi$, the element $\varphi'\in\Bir(X)$ satisfies $\pi'\varphi'=\pi'$. Consider the following diagram:
\[\xymatrix@C+1pc{
	Z\ar@{-->}[rr]^{\text{MMP}/\PP^2}\ar[d]_{\sigma}\ar[drr]^{\eta_Z} && Z'\ar[d]^{\eta_{Z'}} \\ 
	X\ar@{-->}[rr]_{\pi'} && \PP^2
}\]
where $\sigma$ is a resolution of indeterminacy points of $\pi'$, and $Z\dashrightarrow Z'$ is a Minimal Model Program over $\PP^2$. As the general fibres of $Z\to \PP^2$ are rational curves, the morphism $Z'\to \PP^2$ is a Mori fibre space which is a conic bundle. Denote by $\mu: X\dashrightarrow Z'$ the induced birational map and put $\varphi''=\mu\varphi '\mu^{-1}$. Then $\varphi''\in\Bir(Z'/\PP^2)$. By Lemma~\ref{lem: relative Sarkisov}\ref{relative Sarkisov1}, the map $\varphi''$ can be decomposed into the product of Sarkisov links (of conic bundles over surfaces dominating $\PP^2$) of type \II\ and thus into a product of Sarkisov links with no Bertini links. This implies that $\psi^{-1} \varphi\psi$ is in the kernel of the group homomorphism $\Bir(X)\to\bigast\limits_{\Bertini_X^{\ge g}}\ZZ/2$, as desired.

Denoting as in the statement of the theorem  by $G\subseteq \Bir_\CC(\PP^3)$ the group generated by all birational maps $\varphi\in \Bir_\CC(\PP^3)$ such that $\pi\varphi=\pi$ for some rational fibration $\pi\colon \PP^3\dasharrow \PP^2$, we have proven that $G\subseteq \Ker\rho$. The inclusion $\Aut(\PP^3)\subseteq G$ follows from the fact that $\Aut(\PP^3)\simeq \PGL_4(\CC)$ is generated by elementary matrices and that all of them preserve a rational fibration. We moreover have $\Aut(\PP^3)\subsetneq G$ by taking for instance the map
	
\[
\varphi\colon [x_0:x_1:x_2:x_3]\mapsto \left [x_0\frac{x_1}{x_2}:x_1:x_2:x_3 \right ]\in G\setminus \Aut(\PP^3),
\]
which preserves the rational fibration $\pi: [x_0:x_1:x_2:x_3]\dashrightarrow [x_1:x_2:x_3]$. It remains to prove that we may assume that $G\subsetneq \Ker\rho$. To do this, we may simply change $\rho$ by considering an infinite subset $I\subsetneq\mathbb{N}$ and projecting onto $\bigast\limits_{I}\ZZ/2$. 
\end{proof}
\begin{lem}\label{Lem:NormalSubAutP2}
Every normal subgroup $N$ of $\Bir(\PP^2)$ that contains a non-trivial element of $\Aut(\PP^2)$ is equal to $\Bir(\PP^2)$.
\end{lem}
\begin{proof}
As $N\cap \Aut(\PP^2)$ is a non-trivial normal subgroup of the simple group $\Aut(\PP^2)=\PGL_3(\kk)=\PSL_3(\kk)$, it is equal to $\Aut(\PP^2)$.
The Noether-Castelnuovo theorem \cite{Cas} implies that $\Bir(\PP^2)$ is generated by $\Aut(\PP^2)$ and by the standard quadratic transformation $\sigma\colon [x:y:z]\dasharrow [yz:xz:xy]$. As $\sigma$ corresponds locally to $(x,y)\mapsto (x^{-1},y^{-1})$, conjugate to $(x,y)\mapsto (-x,-y)$ via 
\[
(x,y)\mapsto \left(\frac{x+1}{x-1},\frac{y+1}{y-1}\right),
\] we find that $\sigma\in N$. This shows that $N=\Bir(\PP^2)$.
\end{proof}
\begin{lem}\label{Lem:GaGmAutP3}
Every algebraic subgroup of $\Bir(\PP^3)$ that is isomorphic to $\mathbb{G}_a$ or $\mathbb{G}_m$ is conjugate, in $\Bir(\PP^3)$, to a subgroup of $\Aut(\PP^3)$.
\end{lem}
\begin{proof}The result follows from \cite[2.5.8]{BFT}.
The idea of the proof is classical and as follows: one may find a birational map $\PP^3\dasharrow X$, where $X$ is a smooth variety, that conjugates the group $G\simeq \mathbb{G}_a$ or $G\simeq \mathbb{G}_m$ to a subgroup of $\Aut(X)$ (using a result of Weil, see \cite{Weil} or \cite{Kraft}), and then one uses some arguments of Rosenlicht to find an open set $U\subset X$ together with an equivariant isomorphism $U\cong G\times V$, where the action on $V$ is trivial and the action on $G$ is the group multiplication. As $V$ is unirational and of dimension $2$, it is rational. So we can assume that $V=\mathbb{A}^2$ and then obtain a linear action on $\mathbb{A}^3$, that we can extend to $\PP^3$. Note that such an argument is also contained in the work of Bia\l{}ynicki-Birula, see the proof of \cite[Theorem 2]{Bia73}.
\end{proof}

\begin{proof}[Proof of Theorem~\ref{Theorem:AlgSubgroups}]
As we already noticed in the introduction, the subgroup $H_3\subseteq\Bir_{\CC}(\PP^3)$ generated by all connected algebraic subgroups of $\Bir_\CC(\PP^3)$ is a normal subgroup of  $\Bir_{\CC}(\PP^3)$. It remains to prove the main part of Theorem~\ref{Theorem:AlgSubgroups}, i.e.~to prove that $H_3\not=\Bir(\PP^3)$. To do this, it suffices to show that $H_3$ is generated by elements $\varphi\in H_3$ that are conjugate in $\Bir(\PP^3)$ to elements of $\Aut(\PP^3)$ and to apply Theorem~\ref{thm:TameproblemBirP3}. Every connected algebraic subgroup $G\subset \Bir(\PP^3)$ is a linear algebraic group \cite[Remark 2.21]{BF13} and is thus generated by its algebraic subgroups isomorphic to $\mathbb{G}_m$ or to $\mathbb{G}_a$. The proof ends by  applying Lemma~\ref{Lem:GaGmAutP3}.
\end{proof}
\def\bibindent{2.5em}

\end{document}